\documentclass[a4paper,12pt,twoside]{amsart}

\usepackage[utf8]{inputenc}
\usepackage{amsmath}
\usepackage{amssymb}
\usepackage{latexsym}
\usepackage{amsthm}
\usepackage{bbm}
\usepackage{epsfig}
\usepackage{graphicx}
\usepackage{xspace}
\usepackage{caption}
\usepackage{subfig}
\usepackage[autostyle]{csquotes}
\usepackage{xcolor}
\usepackage{hyperref}
\usepackage{url, hypcap}
\definecolor{darkblue}{RGB}{0,0,160}
\hypersetup{
	colorlinks,%
	citecolor=darkblue,%
	filecolor=black,%
	linkcolor=darkblue,%
	urlcolor=darkblue
}
\usepackage[
text={440pt,575pt},
headheight=9pt,
centering
]{geometry}

\newcommand{\bm}[1]{\mbox{\boldmath{$#1$}}}
\newcommand{\bdgr}[1]{\bm{#1}}
\newcommand{\sbdgr}[1]{\scalebox{.7}{\bm{#1}}}

\newtheorem{thm}{Theorem}[section]
\newtheorem{lemma}[thm]{Lemma}

\newtheorem{cor}[thm]{Corollary}

\theoremstyle{definition}

 \numberwithin{equation}{section}
\newcommand{\Mathematica}{\textsc{Mathematica}\xspace}

\begin{document}
\title[$ D $-optimal Designs for Poisson Regression with Synergetic Interaction]{$ D $-optimal Designs for Poisson Regression with Synergetic Interaction Effect}

\author[F.~Freise]{Fritjof Freise}
\address[]{Department of Biometry, Epidemiology and Information Processing\\
	University of Veterinary Medicine Hannover\\Bünteweg 2,
	30559 Hannover\\Germany}
\email{fritjof.freise@tiho-hannover.de}
\author[U.~Graßhoff]{Ulrike Graßhoff}
\address[]{School of Business and Economics\\ Humboldt-University Berlin\\ Unter den Linden~6, 10099 Berlin\\ Germany}
\email{grasshou@hu-berlin.de}
\author[F.~Röttger]{Frank Röttger}
\address[]{MPI MiS Leipzig\\Inselstraße 22\\04103 Leipzig\\Germany}
\email{frank.rottger@mis.mpg.de}
\urladdr{\url{https://sites.google.com/view/roettger}}
\thanks{Corresponding author: Frank Röttger}
\author[R.~Schwabe]{Rainer Schwabe}
\address[]{Institute for Mathematical Stochastics\\ Otto-von-Guericke-University
	Magdeburg\\ Universit\"ats\-platz 2, 39106 Magdeburg\\Germany}
\email{rainer.schwabe@ovgu.de}
\urladdr{\url{http://www.imst3.ovgu.de}}

\begin{abstract}
	We characterize $D$-optimal designs in the two-dimensional Poisson regression model with synergetic interaction and provide an explicit proof.
	The proof is based on the idea of reparameterization of the design region in terms of contours of constant intensity.
	This approach leads to a substantial reduction of complexity as properties of the sensitivity can be treated along and across the contours separately.
	Furthermore, some extensions of this result to higher dimensions are presented.
	
\end{abstract}

\maketitle	

\section{Introduction}

Count data plays an important role in medical and pharmaceutical development, marketing, or psychological research. 
For example, Vives, Losilla, and Rodrigo \cite{VLR2006} performed a review on articles published in psychological journals in the period from 2002 to 2006.
There they found out that a substantial part of these articles dealt with count data for which the mean was quite low (for details we refer to the discussion in Graßhoff et al.\ \cite{GHS2020}).
In these situations, standard linear models are not applicable because they cannot account for the inherent heteroscedasticity.
Instead Poisson regression models are often more appropriate to describe such data. 
As an early source in psychological research we may refer to the Rasch Poisson counts model introduced by Rasch \cite{Rasc60} in 1960 to predict person ability in an item response setup.

The Poisson regression model can be considered as a particular Generalized Linear Model (see McCullagh and Nelder \cite{MN1989}). 
For the analysis of count data in the Poisson regression model there is a variety of literature (see e.\,g.\ Cameron and Trivedi \cite{CT2013}) and the statistical analysis is implemented in main standard statistical software packages (cf.\ ``glm" in R,``GENLIN" in SPSS, ``proc genmod" in SAS),
But only few work has been done to design such experiments.
Ford, Torsney and Wu derived optimal designs for the one-dimensional Poisson regression model in their pioneering paper on canonical transformations  \cite{FTW92}.
Wang et al.\ \cite{WMSY06} obtained numerical solutions for optimal designs in two-dimensional Poisson regression models both for the main effects only (additive) model as well as for the model with interaction term.
For the main effects only model the optimality of their design was proven analytically by Russell et al.\ \cite{RWL09} even for larger dimensions.
Rodr\'{\i}guez{-}Torreblanca and Rodr\'{\i}guez{-}D\'{\i}az \cite{RR07} extended the result by Ford et al.\ for one-dimensional Poisson regression to overdispersed data specified by a negative binomial regression model, and Schmidt and Schwabe \cite{SS17} generalized the result by Russell et al.\ for higher-dimensional Poisson regression to a much broader class of additive regression models.
Graßhoff et al.\ \cite{GHS2020} gave a complete characterization of optimal designs in an ANOVA-type setting for Poisson regression with binary predictors and Kahle et al.\ \cite{KOS2015} indicate, how interactions could be incorporated in this particular situation.

In this paper, we find $ D $-optimal designs for the two-dimensional Poisson regression model with synergetic interaction as before considered numerically by Wang et al.~\cite{WMSY06}. 
We show the $ D $-optimality by reparametrizing the design space via hyperbolic coordinates, such that the inequalities in the Kiefer--Wolfowitz equivalence theorem only need to be checked on the boundary and the diagonal of the design region. This allows us to find an analytical proof for the $ D $-optimality of the proposed design.
Furthermore, we extend this result in various ways to higher-dimensional Poisson regression. First, we find $ D $-optimal designs for first-order and second-order interactions, given that the prespecified interaction parameters are zero. 
Second, we present a $ D $-optimal design for Poisson regression with first-order synergetic interaction where the design space is restricted to the union of the two-dimensional faces of the positive orthant.

The paper is organized as follows.
In the next section we introduce the basic notations for Poisson regression models and specify the corresponding concepts of information and design in Section~\ref{s:Information}.
Results for two-dimensional Poisson regression with interaction are established in Section~\ref{s:main}. 
In Section~\ref{s:k-dim}, we present some extensions to higher-dimensional Poisson regression models.
Further extensions are discussed in Section~\ref{s:discussion}.
Technical proofs have been deferred to an Appendix. 
We note that most of the inequalities there have first been detected by using the computer algebra system \Mathematica~\cite{Mathematica}, but analytical proofs are provided in the Appendix for the readers' convenience. 

\section{Model Specification}

We consider the Poisson regression model where observations $Y$ are Poisson distributed with intensity $E(Y)=\lambda(\mathbf{x})$ which depends on one or more explanatory variables 
$\mathbf{x}=(x_1,...,x_k)$ in terms of a generalized linear model.
In particular, we assume a log-link which relates the mean $\lambda(\mathbf{x})$ to a linear component $\mathbf{f}(\mathbf{x})^\top \bdgr{\beta}$ by
$\lambda(\mathbf{x}) = \exp(\mathbf{f}(\mathbf{x})^\top \bdgr{\beta})$,
where $\mathbf{f}(\mathbf{x})=(f_1(\mathbf{x}),...,f_p(\mathbf{x}))^\top$ is a vector of $p$ known regression functions and $\bdgr{\beta}$ is a $p$-dimensional vector of unknown parameters. 
For example, if $\mathbf{x}=x$ is one-dimensional ($k=1$), then simple Poisson regression is given by $\mathbf{f}(x)=(1,x)^\top$ with $p=2$, $\bdgr{\beta}=(\beta_0,\beta_1)^\top$ and intensity $\lambda(x)=\exp(\beta_0+\beta_1 x)$.
For two explanatory variables $\mathbf{x}=(x_1,x_2)$ ($k=2$) multiple Poisson regression without interaction is given by $\mathbf{f}(\mathbf{x})=(1,x_1,x_2)^\top$ with $p=3$, $\bdgr{\beta}=(\beta_0,\beta_1,\beta_2)^\top$ and intensity $\lambda(\mathbf{x})=\exp(\beta_0+\beta_1 x_1+\beta_2 x_2)$.

In what follows we will focus on the two-dimensional multiple regression ($\mathbf{x}=(x_1,x_2)$, $k=2$) with interaction term, where $p=4$, $\mathbf{f}(\mathbf{x})=(1,x_1,x_2,x_1 x_2)^\top$, $\bdgr{\beta}=(\beta_0,\beta_1,\beta_2,\beta_{12})^\top$ and intensity 
 \begin{equation}
 \label{ModW}
\lambda(\mathbf{x})=\exp(\beta_0+\beta_1 x_1+\beta_2 x_2+\beta_{12} x_1 x_2).
\end{equation}
Here $\beta_0$ is an intercept term such that the mean is $\exp(\beta_0)$ when the explanatory variables are equal to $0$. The quantities $\beta_1$ and $\beta_2$ denote the direct effects of each single explanatory variable, and $\beta_{12}$ describes the amount of the interaction effect when both explanatory variables are active (non-zero).

Typically the explanatory variables describe non-negative quantities ($x_1, x_2 \geq 0$) like doses of some chemical or pharmaceutical agents --- or difficulties of tasks in item response experiments in psychology. 
In particular, in the latter case the expected number of counts (correct answers) decreases with increasing difficulty. 
Then it is reasonable to assume that the direct effects are negative ($\beta_1,\beta_2 < 0$), and that the interaction effect tends into the same direction if present ($\beta_{12} \leq 0$).
In the case that $\beta_{12} < 0$ this will be called a synergy effect because it describes a strengthening of the effect if both components are used simultaneously.

\section{Information and Design}\label{s:Information}

In experimental situations the setting $\mathbf{x}$ of the explanatory variables may be chosen by the experimenter from some experimental region $\mathcal{X}$. 
As the explanatory variables describe non-negative quantities, and if there are no further restrictions on these quantities, it is natural to assume that the design region $\mathcal{X}$ is the non-negative half-axis $[0,\infty)$ or the closure of quadrant I in the Cartesian plane, $[0,\infty)^2$, in one- or two-dimensional Poisson regression, respectively.

To measure the contribution of an observation $Y$ at setting $\mathbf{x}$ the corresponding information can be used: 
With the log-link the Poisson regression model  constitutes a generalized linear model with canonical link \cite{MN1989}. Furthermore for Poisson distributed observations $Y$ the variance and the mean coincide, $\text{Var}(Y)=\mathbb{E}(Y)=\lambda(\mathbf{x})$.
Hence, according to \cite{Atkinson2014} the elemental (Fisher) information for an observation $Y$ at a setting $\mathbf{x}$ is a $p \times p$ matrix given by
\[
\mathbf{M}_{\sbdgr{\beta}}(\mathbf{x}) = \lambda(\mathbf{x})\mathbf{f}(\mathbf{x}) \mathbf{f}(\mathbf{x})^{\top}.
\]
Note that on the right-hand side the intensity $\lambda(\mathbf{x})=\exp(\mathbf{f}(\mathbf{x})^\top \bdgr{\beta})$ depends on the linear component $\mathbf{f}(\mathbf{x})^\top \bdgr{\beta}$ and, hence, on the parameter vector $\bdgr{\beta}$.
Consequently also the information depends on $\bdgr{\beta}$ as indicated by the notation $\mathbf{M}_{\sbdgr{\beta}}$.

For $N$ independent observations $Y_1,...,Y_N$ at settings $\mathbf{x}_1,...,\mathbf{x}_N$ the joint Fisher information matrix is obtained as the sum of the elemental information matrices,
\[
\mathbf{M}_{\sbdgr{\beta}}(\mathbf{x}_1,...,\mathbf{x}_N) = \sum_{i=1}^N \lambda(\mathbf{x}_i)\mathbf{f}(\mathbf{x}_i) \mathbf{f}(\mathbf{x}_i)^{\top}.
\]
The collection $\mathbf{x}_1,...,\mathbf{x}_N$ of settings is called an exact design, and the aim of design optimization is to choose these settings such that the statistical analysis is improved. 
The quality of a design can be measured in terms of the information matrix because its inverse is proportional to the asymptotic covariance matrix of the maximum-likelihood estimator of $\bdgr{\beta}$, see Fahrmeir and Kaufmann \cite{FK85}.
Hence, larger information means higher precision. 
However, matrices are not comparable in general.
Therefore one has to confine oneself to some real valued criterion function applied to the information matrix.
In accordance with the literature we will use the most popular $D$-criterion which aims at maximizing the determinant of the information matrix.
This criterion has nice analytical properties and can be interpreted in terms of minimization of the volume of the asymptotic confidence ellipsoid for $\bdgr{\beta}$ based on the maximum-likelihood estimator.
The optimal design will depend on the parameter vector $\bdgr{\beta}$ and is, hence, only locally optimal. 

Finding an optimal exact design is a discrete optimization problem which is often too hard for analytical solutions.
Therefore we adopt the concept of approximate designs in the spirit of Kiefer \cite{Kief74}. 
An approximate design $\xi$ is defined as a collection $\mathbf{x}_0,...,\mathbf{x}_{n-1}$ of $n$ mutually distinct settings in the design region $\mathcal{X}$ with corresponding weights $ w_0,...,w_{n-1} \geq 0$ satisfying $\sum_{i=0}^{n-1} w_i = 1$.
Then an exact design can be written as an approximate design, where $\mathbf{x}_0,...,\mathbf{x}_{n-1}$ are the mutually distinct settings in the exact design with corresponding numbers $N_0,...,N_{n-1}$ of replications, $\sum_{i=0}^{n-1} N_i = N$, and frequencies $w_i = N_i / N$, $i=0,...,{n-1}$.
However, in an approximate design the weights are relaxed from multiples of $1 / N$ to non-negative real numbers which allow for continuous optimization.

For an approximate design $\xi$ the information matrix is defined as
\[
\mathbf{M}_{\sbdgr{\beta}}(\xi) =  \sum_{i=0}^{n-1} w_i \lambda(\mathbf{x}_i) \mathbf{f}(\mathbf{x}_i) \mathbf{f}(\mathbf{x}_i)^\top ,
\]
which therefore coincides with the standardized (per observation) information matrix $\frac{1}{N}\mathbf{M}_{\sbdgr{\beta}}(\mathbf{x}_1,...,\mathbf{x}_N)$.
An approximate design $\xi^*$ will be called locally $D$-optimal at $\bdgr{\beta}$ if it maximizes the determinant of the information matrix $\mathbf{M}_{\sbdgr{\beta}}(\xi)$.

\section{Optimal Designs}\label{s:main}

We start with quoting results from the literature for one-dimensional and two-dimensional regression without interaction:
In the case of one-dimensional Poisson regression the design $\xi_{\beta_1}^*$ which assigns equal weights $w_0^*=w_1^*=1/2$ to the two settings $x_0^*=0$ and $x_1^*=2/|\beta_1|$ is locally $D$-optimal at $\bdgr{\beta}$ on $\mathcal{X}=[0,\infty)$ for $\beta_1<0$, see Rodríguez-Torreblanca and Rodríguez-Díaz \cite{RR07}.

In the case of two-dimensional Poisson regression without interaction the design $\xi_{\beta_1,\beta_2}^*$ which assigns equal weights $w_0^*=w_1^*=w_2^*=1/3$ to the three settings $\mathbf{x}_0^*=(0,0)$, $\mathbf{x}_1^*=(2/|\beta_1|,0)$, and $\mathbf{x}_2^*=(0,2/|\beta_2|)$ is locally $D$-optimal at $\bdgr{\beta}$ on $\mathcal{X}=[0,\infty)^2$ for $\beta_1,\beta_2<0$, see Russell et al.~\cite{RWL09}. 
Note that the optimal coordinates on the axes coincide with the optimal values in the one-dimensional case, see Schmidt and Schwabe \cite{SS17}. 

In both cases the optimal design is minimally supported, i.e.~the number $n$ of support points of the design is equal to the number $p$ of parameters. 
It is well-known that for $D$-optimal minimally supported designs the optimal weights are all equal, $w_i^*=1/p$, see Silvey \cite{Silv80}.
Such optimal designs are attractive as they can be realized as exact designs when the sample size $N$ is a multiple of the number of parameters $p$.

Further note that these optimal designs always include the setting $x_0=0$ or $\mathbf{x}_0=(0,0)$, respectively, where the intensity $\lambda$ attains its largest value.

The above findings coincide with the numerical results obtained by Wang et al.~\cite{WMSY06} who also numerically found minimally supported $D$-optimal designs for the case of two-dimensional Poisson regression with interaction.
In what follows we will give explicit formulae for these designs and establish rigorous analytical proofs of their optimality.

We start with the special situation of vanishing interaction ($\beta_{12}=0$).
In this case standard methods of factorization can be applied to establish the optimal design, see Schwabe \cite{Schw96}, section 4.

\begin{thm}
\label{thm:prod}
	  If $\beta_1,\beta_2<0$ and $\beta_{12}=0$, then the design $\xi_{\beta_1}^*\otimes\xi_{\beta_2}^*$ which assigns equal weights $w_0^*=w_1^*=w_2^*=w_3^*=1/4$ to the four settings $\mathbf{x}_0^*=(0,0)$, $\mathbf{x}_1^*=(2/|\beta_1|,0)$, $\mathbf{x}_2^*=(0,2/|\beta_2|)$, and $\mathbf{x}_3^*=(2/|\beta_1|,2/|\beta_2|)$ is locally $D$-optimal at $\bdgr{\beta}$ on $\mathcal{X}=[0,\infty)^2$.
\end{thm}

\begin{proof}
	 The regression function $\mathbf{f}(\mathbf{x})=(1,x_1,x_2,x_1 x_2)^\top$ is the Kronecker product of the regression functions $\mathbf{f}_1(x_1)=(1,x_1)^\top$ and $\mathbf{f}_1(x_2)=(1,x_2)^\top$ in the corresponding marginal one-dimensional Poisson regression models, and the design region $\mathcal{X}$ is the Cartesian product of the marginal design regions $\mathcal{X}_1=\mathcal{X}_2=[0,\infty)$. 
	 Also the intensity $\lambda(\mathbf{x})=\exp(\beta_0+\beta_1 x_1+\beta_2 x_2)$ factorizes into the marginal intensities $\lambda_1(x_1)=\exp(\beta_0+\beta_1 x_1)$ and $\lambda_2(x_2)=\exp(\beta_2 x_2)$ for the marginal parameters $\bdgr{\beta}_1=(\beta_0,\beta_1)^\top$ and $\bdgr{\beta}_2=(0,\beta_2)^\top$, respectively.
	  As mentioned before the designs $\xi_{\beta_j}^*$ which assign equal weights $1/2$ to the settings $x_{j0}=0$ and $x_{j1}=2/|\beta_j|$ are locally $D$-optimal at $\bdgr{\beta}_j$ on $\mathcal{X}_j$, $j=1,2$.
	  Then the product type design $\xi_{\beta_1}^*\otimes\xi_{\beta_2}^*$ which is defined as the measure theoretic product of the marginals is locally $D$-optimal at $\bdgr{\beta}$ by an application of Theorem~4.2 in \cite{Schw96}.
\end{proof}

In contrast to the result of Theorem~\ref{thm:prod} the intensity fails to factorize in the case of a non-vanishing interaction ($\beta_{12} \neq 0$). Thus a different approach has to be chosen.
As a prerequisite we mention that in the above cases the optimal designs can be derived from those for standard parameter values $\beta_0=0$ and $\beta_1=-1$ in one dimension or $\beta_1=\beta_2=-1$ in two dimensions by canonical transformations, see Ford et al.~\cite{FTW92}, or, more generally, by equivariance considerations, see Radloff and Schwabe \cite{RS2016}.
We will adopt this approach also to the two-dimensional Poisson regression model with interaction and consider the case $\beta_0=0$ and $\beta_1=\beta_2=-1$ first.
There the interaction effect remains a free parameter, and we denote the strength of the synergy effect by $\rho=-\beta_{12} \geq 0$. 

\subsection{Standardized Case}\label{s:standard}

Throughout this subsection we assume the standardized situation with $\bdgr{\beta}=(0,-1,-1,-\rho)^\top$ for some $\rho \geq 0$.
Motivated by Theorem~\ref{thm:prod} and the numerical results in Wang et al.~ \cite{WMSY06} we consider a class $\Xi_0$ of minimally supported designs as potential candidates for being optimal.
In the class $\Xi_0$ the designs have one setting at the origin $\mathbf{x}_0=(0,0)$, where the intensity is highest, one setting $\mathbf{x}_1=(x_1,0)$ and $\mathbf{x}_2=(0,x_2)$ on each of the bounding axes of the design region as for the optimal design in the model without interaction, and an additional setting $\mathbf{x}_3=(t,t)$ on the diagonal of the design region, where the effects of the two components are equal. 
The following result is due to Könner~\cite{koenner}.

\begin{lemma}\label{l:design}
	Let $t=(\sqrt{1+8\rho}-1)/(2\rho)$ for $\rho > 0$ and  $t=2$ for $\rho = 0$.
	Then the design $\xi_t$ which assigns equal weights $1/4$ to $\mathbf{x}_0=(0,0)$, $\mathbf{x}_1=(2,0)$, $\mathbf{x}_2=(0,2)$, and $\mathbf{x}_3=(t,t)$ is	locally $D$-optimal within the class $\Xi_0$. 
\end{lemma} 

Note that $t=2$ for $\rho = 0$ is in accordance with the optimal product-type design in Theorem~\ref{thm:prod}, $t$ is continuously decreasing in $\rho$, and $t$ tends to $0$ when the strength of synergy $\rho$ gets arbitrarily large.
Figure~\ref{fig:t} shows the value of $ t $ in dependence on $ \rho $.
\begin{figure}
	\includegraphics[scale=1]{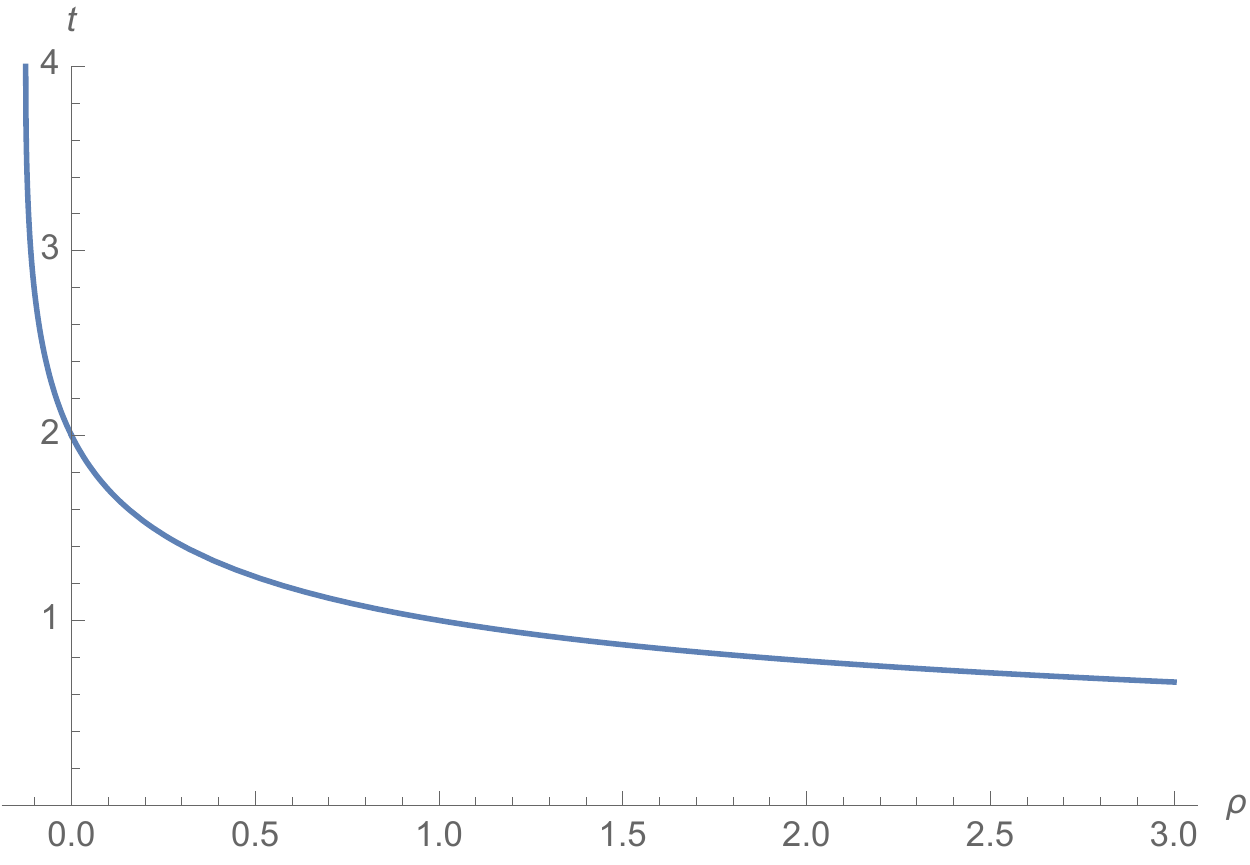}
	\caption{Value of optimal $ t $ in Lemma~\ref{l:design} for $ -1/8\le \rho \le 3 $}\label{fig:t}
\end{figure}

To establish that $\xi_t$ is locally $D$-optimal within the class of all designs on $\mathcal{X}$ we will make use of the Kiefer--Wolfowitz equivalence theorem \cite{kieferwolfowitz1960} in its extended version incorporating intensities, see Fedorov \cite{Fedorov1972}.
For this we introduce the sensitivity function
$
\psi(\mathbf{x};\xi) = \lambda(\mathbf{x})  \mathbf{f}(\mathbf{x})^{\top} \mathbf{M}(\xi)^{-1} \mathbf{f}(\mathbf{x}) ,
$
where we suppress the dependence on $\bdgr{\beta}$ in the notation.
Then by the equivalence theorem a design $\xi^*$ is (locally) $D$-optimal if (and only if) the sensitivity function $\psi(\mathbf{x};\xi^*)$ does not exceed the number $p$ of parameters uniformly on the design region $\mathcal{X}$.
Equivalently we may consider the deduced sensitivity function
 \begin{align*}
 d(\mathbf{x};\xi) 
& = \mathbf{f}(\mathbf{x})^{\top} \mathbf{M}(\xi)^{-1} \mathbf{f}(\mathbf{x})/p - 1/\lambda(\mathbf{x}) 
 \end{align*}
as $ \lambda(\mathbf{x})>0 $.
Then $\xi_t$ is $D$-optimal if $d(\mathbf{x};\xi_t) \leq 0$ for all $\mathbf{x}\in\mathcal{X}$.
To establish this condition we need some preparatory results on the shape of the (deduced) sensitivity function. 
Figure~\ref{fig:deducedsensitivity} shows $ d(\mathbf{x};\xi_t)  $ for $ t=2 $ for $ \rho=0 $, i.e.~for the standardized setting in Theorem~\ref{thm:prod}.
\begin{figure}
	\includegraphics[scale=1]{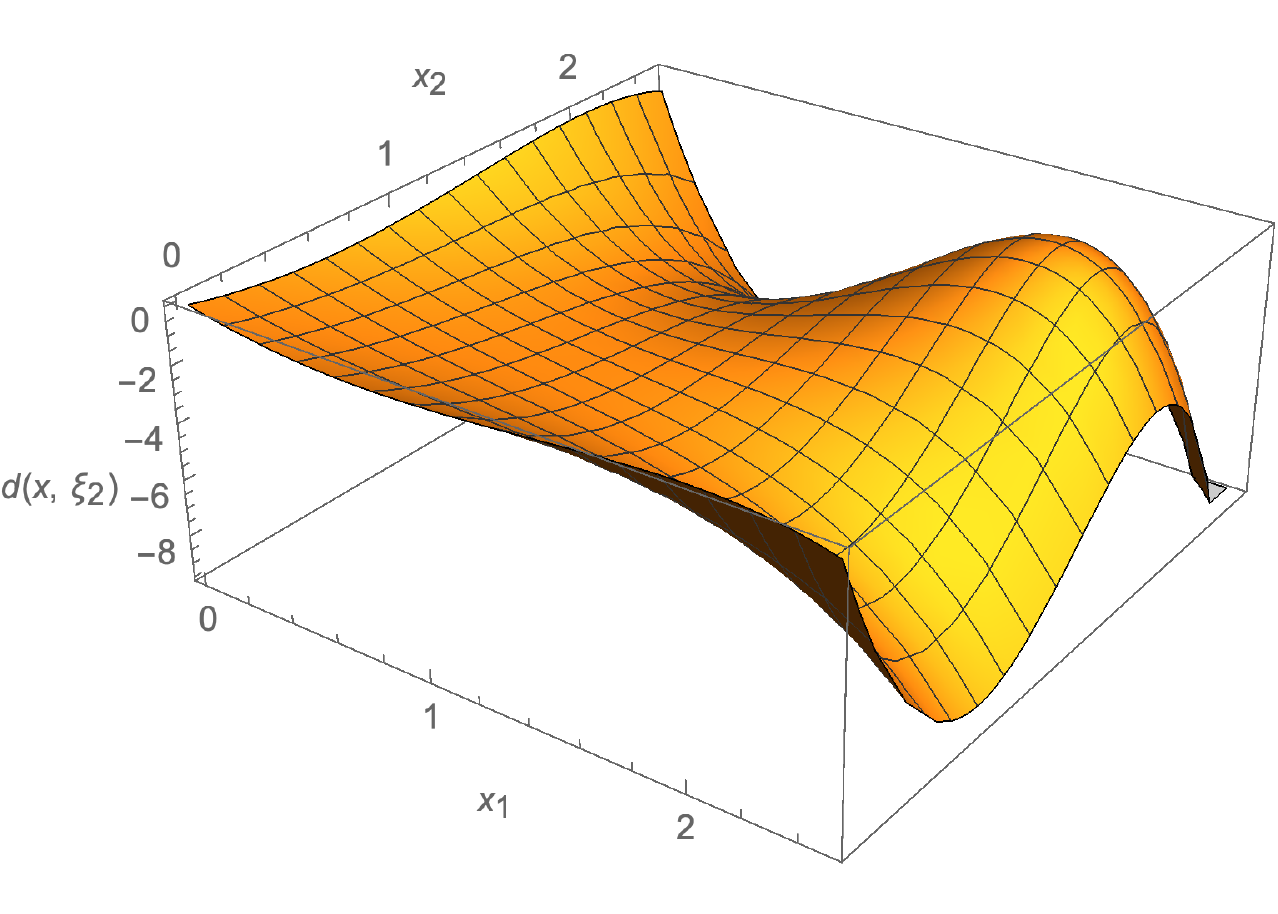}
	\caption{Deduced sensitivity function for $ t=2 $ ($ \rho=0 $)}\label{fig:deducedsensitivity}
\end{figure}

\begin{lemma}
\label{l:hyperbolicpath}
	If $\xi$ is invariant under permutation of $x_1$ and $x_2$, then $d(\mathbf{x};\xi)$ attains its maximum on the boundary or on the diagonal of $\mathcal{X}$.
\end{lemma}

\begin{lemma}
\label{l:boundary}
	$d((x,0);\xi_{t}) = d((0,x);\xi_{t}) \leq 0$ for all $x \geq 0$.
\end{lemma}

\begin{lemma}
\label{l:diagonal}
	$d((x,x);\xi_{t}) \leq 0$ for all $x \geq 0$.
\end{lemma}

Note that $\xi_t$ is invariant with respect to the permutation of $x_1$ and $x_2$.
Then, combining Lemmas~\ref{l:hyperbolicpath} to \ref{l:diagonal}, we obtain $d(\mathbf{x};\xi_{t}) \leq 0$ for all $\mathbf{x} \in \mathcal{X}$ which establishes the $D$-optimality of $\xi_t$ in view of the equivalence theorem.

\begin{thm}\label{t:mainresult}
	 In the two-dimensional Poisson regression model with interaction the design $\xi_t$ is	locally $D$-optimal at $\bdgr{\beta}=(0,-1,-1,-\rho)^\top$ on $\mathcal{X}=[0,\infty)^2$ which assigns equal weights $1/4$ to the $4$ settings $\mathbf{x}_0=(0,0)$, $\mathbf{x}_1=(2,0)$, $\mathbf{x}_2=(0,2)$, and $\mathbf{x}_3=(t,t)$, where $t=(\sqrt{1+8\rho}-1)/(2\rho)$ for $\rho > 0$ and  $t=2$ for $\rho = 0$. 
\end{thm}

\subsection{General case}
\label{s:general}

For the general situation of decreasing intensities ($\beta_1,\beta_2 < 0$) and a synergy effect ($\beta_{12} < 0$) the optimal design can be obtained by simultaneous scaling of the settings $\mathbf{x} = (x_1,x_2) \to \tilde{\mathbf{x}} = (x_1/|\beta_1|,x_2/|\beta_2|)$ and of the parameters $\bdgr{\beta} = (0,-1,-1,-\rho)^\top \to \tilde{\bdgr{\beta}} = (0,\beta_1,\beta_2,-\rho\beta_1\beta_2)^\top$ by equivariance, see Radloff and Schwabe \cite{RS2016}.
This simultaneous scaling leaves the linear component and, hence, the intensity unchanged, $\mathbf{f}(\tilde{\mathbf{x}})^\top \tilde{\bdgr{\beta}} =\mathbf{f}(\mathbf{x})^\top \bdgr{\beta}$.
If the scaling of $\mathbf{x}$ is applied to the settings in $\xi_t$ of Theorem~\ref{t:mainresult}, then the resulting rescaled design will be locally $D$-optimal at $\tilde{\bdgr{\beta}}$ on $\mathcal{X}$ as the design region is invariant with respect to scaling.
Furthermore, the design optimization is not affected by the value $\beta_0$ of the intercept term because this term contributes to the intensity and, hence, to the information matrix only by a multiplicative factor,
$\lambda(\mathbf{x}) = \exp(\beta_0)\exp(\beta_1 x_1 + \beta_2 x_2 + \beta_{12} x_1 x_2)$.
We thus obtain the following result from Theorem~\ref{t:mainresult}.

\begin{thm}
\label{thm:general}
	  Assume the two-dimensional Poisson regression model with interaction and  $\bdgr{\beta}=(\beta_0,\beta_1,\beta_2,\beta_{12})^\top$ with $\beta_1,\beta_2 < 0$ and $\beta_{12} \leq 0$.
	  Let $\rho = - \beta_{12}/(\beta_1\beta_2)$, $t=(\sqrt{1+8\rho}-1)/(2\rho)$ for $\beta_{12} < 0$ and  $t=2$ for $ \beta_{12} = 0$.
	 Then the design which assigns equal weights $1/4$ to the $4$ settings  $\mathbf{x}_0=(0,0)$, $\mathbf{x}_1=(2/|\beta_1|,0)$, $\mathbf{x}_2=(0,2/|\beta_2|)$, and $\mathbf{x}_3=(t/|\beta_1|,t/|\beta_2|)$ is locally $D$-optimal at $\bdgr{\beta}$ on $\mathcal{X}=[0,\infty)^2$.  
\end{thm}

Note that the settings $\mathbf{x}_0$, $\mathbf{x}_1$, and $\mathbf{x}_2$ of the locally $D$-optimal design $\xi_t$ in the model with interaction coincide with those of the optimal design for the model without interaction. 
Only a fourth setting $\mathbf{x}_3=(t/|\beta_1|,t/|\beta_2|)$ has been added in the interior of the design region.

 \section{Higher-dimensional Models}\label{s:k-dim}

In the present section on $k$-dimensional Poisson regression with $k$ explanatory variables ($\mathbf{x}=(x_1,x_2,...,x_k)$, $k \geq 3$) we restrict to the standardized case with zero intercept ($\beta_0=0$) and all main effects $\beta_1=...=\beta_k$ equal to $-1$ for simplicity of notation.
 Extensions to the case of  general $\beta_0$ and $\beta_1,...,\beta_k<0$ can be obtained by the scaling method used for Theorem~\ref{thm:general}.

We first note that for the $k$-dimensional Poisson regression without interactions 
\[
\mathbf{f}(\mathbf{x})^\top \bdgr{\beta} = \beta_0 + \sum_{j=1}^k \beta_j x_j 
\]
Russell et al.~\cite{RWL09} showed that the minimally supported design which assigns equal weights $1/(k+1)$ to the origin $\mathbf{x}_0=(0,...,0)$ and the $k$ axial settings $\mathbf{x}_1=(2,0,...,0)$, $\mathbf{x}_2=(0,2,...,0)$, $...$, $\mathbf{x}_k=(0,...,0,2)$ is locally $D$-optimal at $\bdgr{\beta}=(0,-1,...,-1)^\top$.
Schmidt and Schwabe \cite{SS17} more generally proved that in models without interactions the locally $D$-optimal design points coincide with their counterparts in the marginal one-dimensional models.
This approach will be extended in Theorems~\ref{thm:k-dim_1st-order} and \ref{thm:k-dim_2nd-order} to two- and three-dimensional marginals with interactions.

In what follows we mainly consider the particular situation that all interactions occurring in the models have values equal to $0$ and that the design region is the full orthant $\mathcal{X}=[0,\infty)^k$.
Setting the interactions to zero does not mean that we presume to know that there are no interactions in the model.
Instead we are going to determine locally optimal designs in models with interactions which are locally optimal at such $\bdgr{\beta}$ for which all interaction terms attain the value $0$.

We start with a generalization of Theorem~\ref{thm:prod} to a $k$-dimensional Poisson regression model with complete interactions 
\[
\mathbf{f}(\mathbf{x})^\top \bdgr{\beta} = \beta_0 + \sum_{j=1}^k \beta_j x_j + \sum_{i<j} \beta_{ij} x_i x_j + \quad ... \quad + \beta_{12...k} x_1 x_2 ... x_k ,
\]
where the number of parameters is $p = 2^k$.

\begin{thm}
\label{thm:prod-k}
	  In the $k$-dimensional Poisson regression model with complete interactions the minimally supported design $\xi_{-1}^* \otimes ... \otimes \xi_{-1}^*$ which assigns equal weights $1/p $ to the $p=2^k$ settings of the full factorial on $\{0,2\}^k$ is locally $D$-optimal at $\bdgr{\beta}$ on $\mathcal{X}=[0,\infty)^k$, when $\beta_1=...=\beta_k=-1$ and all interactions $\beta_{ij}, ..., \beta_{12...k}$ are equal to $0$.
\end{thm}

The proof of Theorem~\ref{thm:prod-k} follows the lines of the proof of Theorem~\ref{thm:prod} as all of the design region $\mathcal{X}$, the vector of regression functions $\mathbf{f}$, and the intensity function $\lambda$ factorize to their one-dimensional counterparts.
Hence, details will be omitted.

Now we come back to the Poisson regression model with first-order interactions
\[
\mathbf{f}(\mathbf{x})^\top \bdgr{\beta} = \beta_0 + \sum_{j=1}^k \beta_j x_j + \sum_{i<j} \beta_{ij} x_i x_j ,
\]
where the number of parameters is $p = 1+k+k(k-1)/2$.

\begin{thm}
\label{thm:k-dim_1st-order}
	  In the $k$-dimensional Poisson regression model with first-order interactions the minimally supported design which assigns equal weights $1/p $ to the $p=1+k+k(k-1)/2$ settings $\mathbf{x}_0=(0,0,...,0)$, $\mathbf{x}_1=(2,0,...,0)$, $\mathbf{x}_2=(0,2,...,0)$, $...$, $\mathbf{x}_k=(0,...,0,2)$, and $\mathbf{x}_{ij}=\mathbf{x}_i+\mathbf{x}_j$, $1 \leq i < j \leq k$, is locally $D$-optimal at $\bdgr{\beta}$ on $\mathcal{X}=[0,\infty)^k$, when $\beta_1=...=\beta_k=-1$ and $\beta_{ij} = 0$, $1 \leq i < j \leq k$.  
\end{thm}

For illustrative purposes we specify this result for $k=3$ components.

\begin{cor}\label{ex:3D}
	  In the three-dimensional Poisson regression model with first-order interactions 
\[	  
\mathbf{f}(\mathbf{x})^\top \bdgr{\beta} = \beta_0 + \beta_1 x_1 + \beta_2 x_2 + \beta_3 x_3 + \beta_{12} x_1 x_2 + \beta_{13} x_1 x_3 + \beta_{23} x_2 x_3  
\]
	 the minimally supported design which assigns equal weights $1/7$ to the $7$ settings $\mathbf{x}_0=(0,0,0)$, $\mathbf{x}_1=(2,0,0)$, $\mathbf{x}_2=(0,2,0)$, $\mathbf{x}_3=(0,0,2)$, $\mathbf{x}_4=(2,2,0)$, $\mathbf{x}_5=(2,0,2)$, and $\mathbf{x}_6=(0,2,2)$ is locally $D$-optimal at $\bdgr{\beta}$ on $\mathcal{X}=[0,\infty)^3$, when $\beta_1=\beta_2=\beta_3=-1$ and $\beta_{12} = \beta_{13} = \beta_{23} = 0$.
\end{cor}

\begin{figure}
	\centering
	\includegraphics[scale=.5]{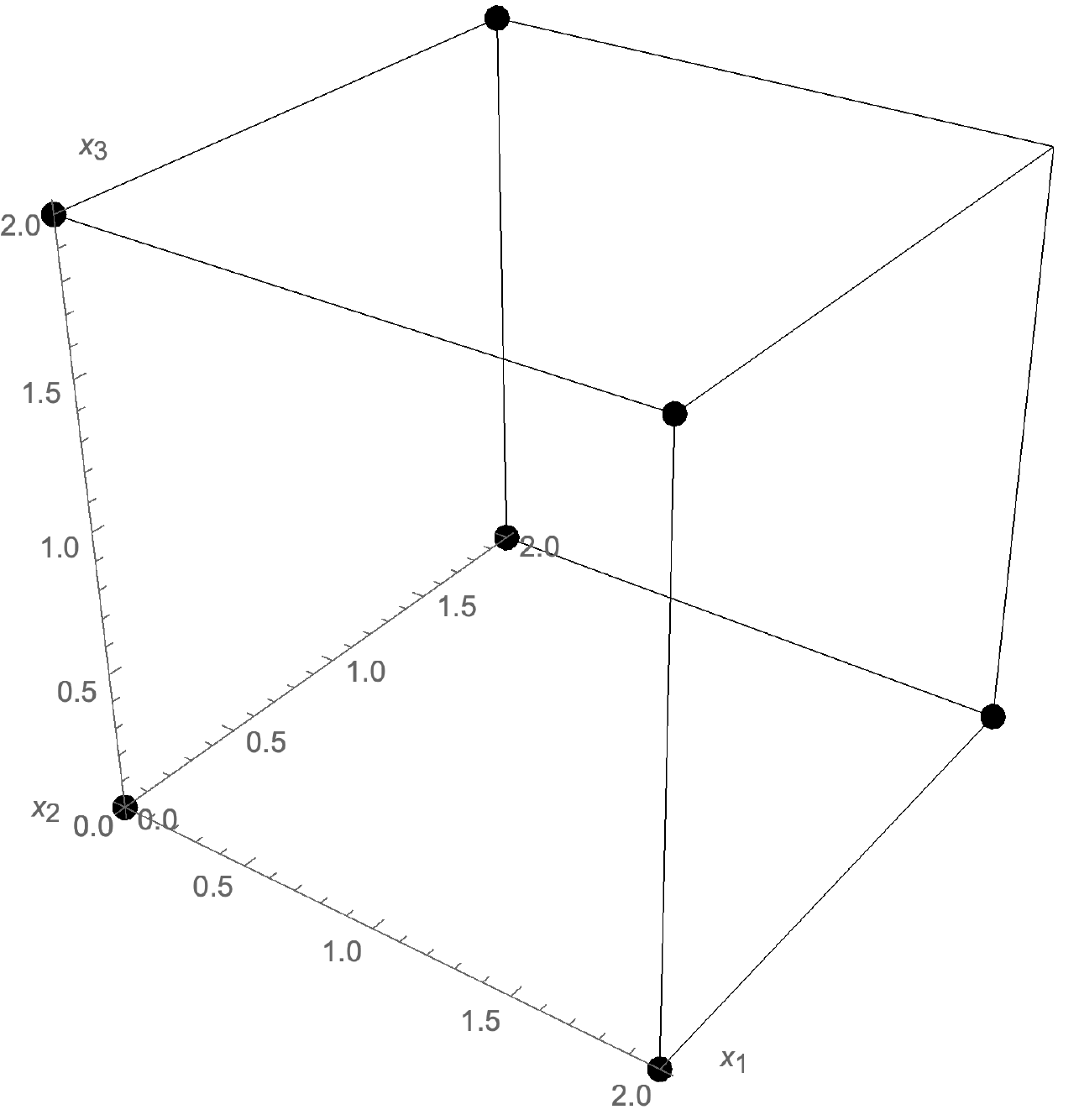}
	\caption{Design points in Example~\ref{ex:3D}}\label{fig:example3D}
\end{figure}

The optimal design points of Corollary~\ref{ex:3D} are visualized in Figure~\ref{fig:example3D}.  
Note that in in the Poisson regression model with first-order interactions the locally $D$-optimal design has only support points on the axes and on the diagonals of the faces, but none in the interior of the design region, and that the support points on each face coincide with the optimal settings for the corresponding two-dimensional marginal model.
Thus only those settings are included from the full factorial $\{0,2\}^k$ of the complete interaction case (Theorem~\ref{thm:prod-k}) which have, at most, two non-zero components, and the locally $D$-optimal design concentrates on settings with higher intensity.
This is in accordance with the findings for the Poisson regression model without interactions, where only those settings will be used which have, at most, one non-zero component, and carries over to higher-order interactions.
In particular, for the Poisson regression model with second-order interactions
\[
\mathbf{f}(\mathbf{x})^\top \bdgr{\beta} = \beta_0 + \sum_{j=1}^k \beta_j x_j + \sum_{i<j} \beta_{ij} x_i x_j + \sum_{i<j<\ell} \beta_{ij\ell} x_i x_j x_\ell ,
\]
where the number of parameters is $p = 1+k+k(k-1)/2+k(k-1)(k-2)/6$, we obtain a similar result.

\begin{thm}
\label{thm:k-dim_2nd-order}
	  In the $k$-dimensional Poisson regression model with second-order interactions the minimally supported design which assigns equal weights $ 1/p $ to the $p=1+k+k(k-1)/2+k(k-1)(k-2)/6$ settings $\mathbf{x}_0=(0,0,...,0)$, $\mathbf{x}_1=(2,0,...,0)$, $\mathbf{x}_2=(0,2,...,0)$, $...$, $\mathbf{x}_k=(0,...,0,2)$, $\mathbf{x}_{ij}=\mathbf{x}_i+\mathbf{x}_j$, $1 \leq i < j \leq k$, and $\mathbf{x}_{ij\ell}=\mathbf{x}_i+\mathbf{x}_j+\mathbf{x}_\ell$, $1 \leq i < j < \ell \leq k$, is locally $D$-optimal at $\bdgr{\beta}$ on $\mathcal{X}=[0,\infty)^k$, when $\beta_1=...=\beta_k=-1$, $\beta_{ij} = 0$, $1 \leq i < j \leq k$, and $\beta_{ij\ell} = 0$, $1 \leq i < j < \ell \leq k$.  
\end{thm}

The proofs of Theorems~\ref{thm:k-dim_1st-order} and \ref{thm:k-dim_2nd-order} are based on symmetry properties which get lost if one or more of the interaction terms are non-zero.
However, if only few components of $\mathbf{x}$ may be active (non-zero), then locally $D$-optimal designs may be obtained in the spirit of the proof of Lemma~\ref{l:boundary} for synergetic interaction effects.
We demonstrate this in the setting of first-order interactions $\rho_{ij}=-\beta_{ij}\geq 0$, when the design region $\mathcal{X}$ consists of the union of the two-dimensional faces of the orthant, i.\,e.\ when, at most, two components of $\mathbf{x}$ can be active.

\begin{thm}
\label{thm:faces}
	  Consider the $k$-dimensional Poisson regression model with first-order interactions on $\mathcal{X}=\bigcup_{i<j}\mathcal{X}_{ij}$, where $\mathcal{X}_{ij}=\{(x_1,...,x_k);\ x_i,x_j \geq 0, x_\ell=0 ~~\text{for}~~ \ell \neq i,j\}$ is the two-dimensional face related to the $i$th and $j$th component.
	  Let $\beta_1=...=\beta_k=-1$, $\rho_{ij} = -\beta_{ij} \geq 0$, $t_{ij}=(\sqrt{1+8\rho_{ij}}-1)/(2\rho_{ij})$ for $\rho_{ij} > 0$, $t_{ij}=2$ for $\rho_{ij} = 0$, and $\mathbf{x}_{ij}\in\mathcal{X}_{ij}$ with $x_i=x_j=t_{ij}$, $1 \leq i < j \leq k$.
	  Then the minimally supported design which assigns equal weights $1/(1+k+k(k-1)/2)$ to the $1+k+k(k-1)/2$ settings $\mathbf{x}_0=(0,0,...,0)$, $\mathbf{x}_1=(2,0,...,0)$, $\mathbf{x}_2=(0,2,...,0)$, $...$, $\mathbf{x}_k=(0,...,0,2)$, and $\mathbf{x}_{ij}$, $1 \leq i < j \leq k$, is locally $D$-optimal at $\bdgr{\beta}$ on $\mathcal{X}$.  
\end{thm}

This result follows as in the proof of Lemma~\ref{l:boundary}. 
We believe that the $ D $-optimality of the design in Theorem~\ref{thm:faces} could also hold on the whole positive orthant if we assume that the prespecified interaction parameters are identical and non-positive. A proof of this statement should follow in the spirit of Farrell et al.~\cite{FKW1967}, similar to the constructions in the Lemmas~\ref{l:hyperbolicpath} and \ref{l:diagonal} and the proof of Theorem~\ref{thm:k-dim_1st-order}.

However, in the situation of general synergy effects an analogon to Lemma~\ref{l:hyperbolicpath} cannot be established because of the lacking symmetry.
Hence, it remains open whether the design of Theorem~\ref{thm:faces} retains its optimality in the general setting.

 \section{Discussion}\label{s:discussion}

The main purpose of the present paper is to characterize locally $D$-optimal designs explicitly for the two-dimensional Poisson regression model with interaction on the unbounded design region of quadrant \textrm{I} when both main effects as well as the interaction effect are negative, and to present a rigorous proof for their optimality.
Obviously the designs specified in Theorem~\ref{thm:general} remain optimal on design regions which are subsets of quadrant \textrm{I} and cover the support points of the respective design.
For example, if the design region is a rectangle, $\mathcal{X} = [0,b_1] \times [0,b_2]$, then the design of Theorem~\ref{thm:general} is optimal as long as $b_1 \geq 2/|\beta_1|$ and $b_2 \geq 2/|\beta_2|$ for the two components.
Furthermore, if the design region is shifted, $\mathcal{X} = [a_1,\infty) \times [a_2,\infty)$ or a sufficiently large subregion of that, then also the locally $D$-optimal design is shifted accordingly and assigns equal weights $1/4$ to $\mathbf{x}_0=(a_1,a_2)$, $\mathbf{x}_1=(a_1+2/|\beta_1|,a_2)$, $\mathbf{x}_2=(a_1,a_2+2/|\beta_2|)$, and $\mathbf{x}_3=(a_1+t/|\beta_1|,a_2+t/|\beta_2|)$ where $t$ is defined as in Theorem~\ref{thm:general}.

Although the locally $D$-optimal designs only differ in the location of the support point on the diagonal, if the main effects are kept fixed, they are quite sensitive with respect to the strength $\rho$ of the synergy parameter in their performance.
The quality of their performance can be measured in terms of the local $D$-efficiency which is defined as $\mathrm{eff}_D(\xi,\bdgr{\beta}) = \left(\det(\mathbf{M}_{\sbdgr{\beta}}(\xi))/\det(\mathbf{M}_{\sbdgr{\beta}}(\xi_{\sbdgr{\beta}}^*))\right)^{(1/p)}$ for a design $\xi$, where $\xi_{\sbdgr{\beta}}^*$ denotes the locally $D$-optimal design at $\bdgr{\beta}$.
This efficiency can be interpreted as the asymptotic proportion of observations required for the locally $D$-optimal $\xi_{\sbdgr{\beta}}^*$ to obtain the same precision as for the competing design $\xi$ of interest.
For example, in the standardized case of Subsection~\ref{s:standard} the design $\xi_x$ would be locally $D$-optimal when the strength of synergy would be $(2-x)/x^2$. 
Its local $D$-efficiency can be calculated as $\mathrm{eff}_D(\xi,\bdgr{\beta}) = (x/t) \exp((2 t + \rho t^2 - 2 x - \rho x^2)/4)$ when $\rho$ is the true strength of synergy and $t$ is the corresponding optimal coordinate on the diagonal ($t=(\sqrt{1+8\rho}-1)/(2\rho)$ for $\rho > 0$ and  $t=2$ for $\rho = 0$).
For selected values of $x$ the local $D$-efficiencies are depicted in Figure~\ref{fig:eff}.
\begin{figure}
	\includegraphics[scale=1]{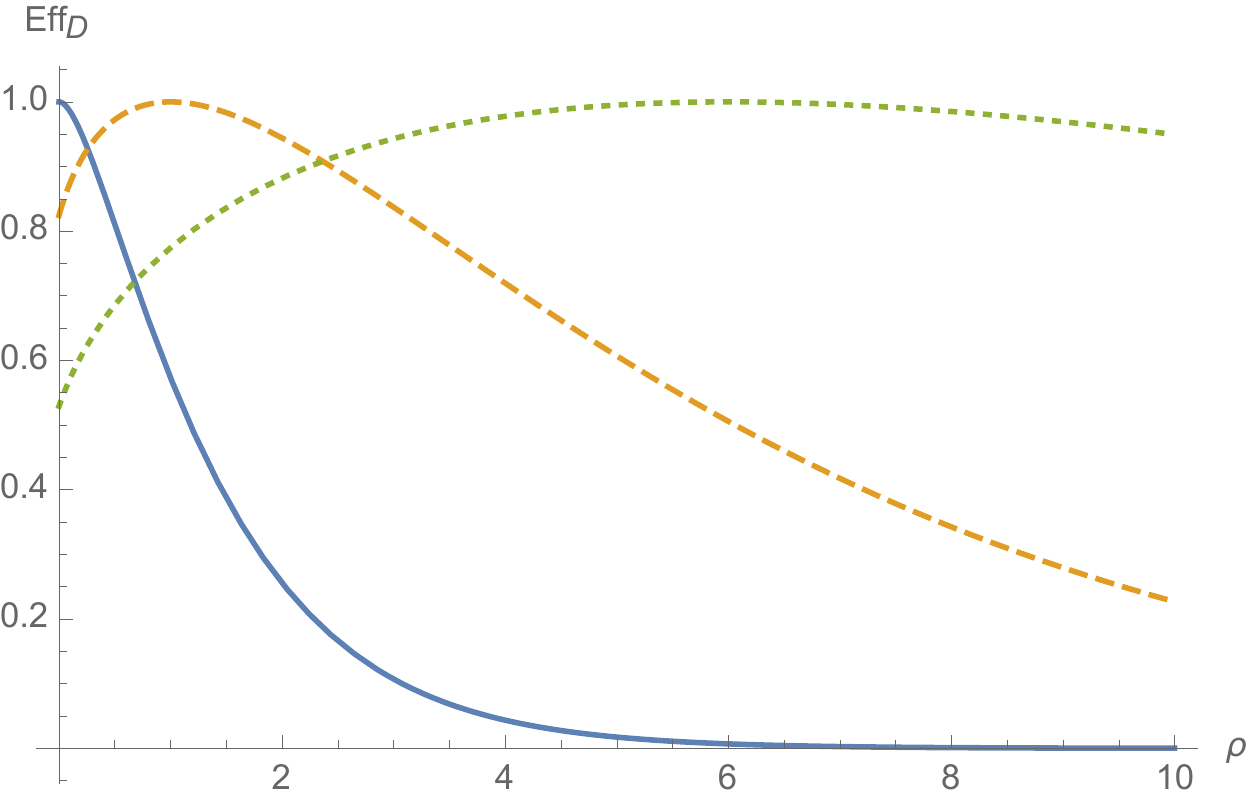}
	\caption{Efficiency of $ \xi_x $ for $ x=2 $ (solid line), $ x=1 $ (dashed) and $ x=1/2 $ (dotted)}\label{fig:eff}
\end{figure}
The appealing product-type design $\xi_2$ of Theorem~\ref{thm:prod} rapidly loses efficiency if the strength $\rho$ of synergy substantially increase.
The triangular design $\xi_1$ seems to be rather robust over a wide range of strength parameters, while for smaller $x$ the design $\xi_x$ loses efficiency when there is no synergy effect ($\rho=0$).
Hence, it would be desirable to determine robust designs like maximin $D$-efficient or weighted (``Bayesian'') optimal designs (see e.\,g.\ Atkinson et al.~\cite{ADT2007}), but this would go beyond the scope of the present paper.

If in contrast to the situation of Theorems~\ref{t:mainresult} and \ref{thm:general} there is an antagonistic interaction effect which means that $\beta_{12}$ is positive ($\rho<0$), no optimal design will exist on quadrant \textrm{I} because the determinant of the information matrix becomes unbounded.
However, if we restrict the design region to a rectangle one may be tempted to extend the above results. 
For example, in the standardized case ($\beta_1=\beta_2=-1$) on a square design region Lemma~\ref{l:design} may be extended as follows

\begin{lemma}\label{l:antagon}
	Let $b\geq 2$, $\rho < 0$, and $t=(\sqrt{1+8\rho}-1)/(2\rho)$ for $\rho > -1/8$.
	\\
(a)	If $\rho > -1/8$, $t \leq b$ and $t^4\exp(-2 t - \rho t^2) \geq b^4\exp(-2 b - \rho b^2)$, then the design $\xi_t$ is	locally $D$-optimal within the class $\Xi_0$ on $\mathcal{X}=[0,b]^2$.
\\
(b)	If $\rho \leq -1/8$ or $b < t$ or $t^4\exp(-2 t - \rho t^2) < b^4\exp(-2 b - \rho b^2)$, then the design $\xi_b$ is	locally $D$-optimal within the class $\Xi_0$ on $\mathcal{X}=[0,b]^2$.
\end{lemma} 

Moreover, Lemma~\ref{l:boundary} does not depend on $\rho$ and, if, additionally, $b \leq 1/|\rho|$, then the argumentation in the proof of Lemma~\ref{l:hyperbolicpath} can be adopted, where now the hyperbolic coordinate system is centered at $(1/|\rho|,1/|\rho|)$ and $v$ is negative (cf.\ the proof below).
However, the inequalities of Lemma~\ref{l:diagonal} are no longer valid, in general. 
In particular, for $\rho$ less than, but close to $-1/8$ the (deduced) sensitivity function of the design $\xi_t$ shows a local minimum at $t$ rather than a maximum which disproves the optimality of $\xi_t$ within the class of all designs on $\mathcal{X}=[0,b]^2$.
In that case an additional fifth support point is required on the diagonal, and also the weights have to be optimized.
So, in the case of an antagonistic interaction effect no general analytic solution can be expected and the numerically obtained optimal designs may become difficult to be realized as exact designs.

For even smaller design regions ($b<2$) design points on the adverse boundaries ($x_1=b$ or $x_2=b$) may occur in the optimal designs, but not in the interior besides the diagonal, both in the synergetic as well as in the antagonistic case.

It seems more promising to extend the present results to negative binomial (Poisson-Gamma) regression which is a popular generalization of Poisson regression which can cope with overdispersion as in Rodr\'{\i}guez{-}Torreblanca and Rodr\'{\i}guez{-}D\'{\i}az \cite{RR07} for one-dimensional regression or in Schmidt and Schwabe \cite{SS17} for multidimensional regression without interaction.
This will be object of further investigation.

\subsection*{Acknowledgements}
We acknowledge that the statement of Lemma \ref{l:design} was originally derived by Dörte~Schnur in her thesis \cite{koenner}.
Part of this work was supported by grants HO\,1286/6, SCHW\,531/15 and 314838170, GRK 2297 MathCoRe of the Deutsche Forschungsgemeinschaft DFG.

 \appendix
 
 \section{Proofs}

\begin{proof}[Proof of Lemma~\ref{l:design}]
For a design $\xi$ with settings $\mathbf{x}_i$ and corresponding weights $w_i$, $i=0,...,n-1$, denote by $\mathbf{F}=(\mathbf{f}(\mathbf{x}_0),...,\mathbf{f}(\mathbf{x}_{n-1}))^{\top}$ the $(n \times p)$-dimensional essential design matrix and by the $(n \times n)$-dimensional diagonal matrices $\mathbf{\Lambda}=\mathrm{diag}(\lambda(\mathbf{x}_0),...,\lambda(\mathbf{x}_{n-1}))$ and $\mathbf{W}=\mathrm{diag}(w_0,...,w_{n-1})$ the intensity and the weight matrix, respectively.
Then the information matrix can be written as
\[
\mathbf{M}(\xi) = \mathbf{F}^{\top} \mathbf{W} \mathbf{\Lambda} \mathbf{F} .
\]

For minimally supported designs the matrices $\mathbf{F}$, $\mathbf{W}$ and $\mathbf{\Lambda}$ are quadratic ($p \times p$) and the determinant of the information matrix factorizes,
\[
\det(\mathbf{M}(\xi)) = \det(\mathbf{W}) \det(\mathbf{\Lambda}) \det(\mathbf{F})^2 .
\]
As $\mathbf{W}$ and $\mathbf{\Lambda}$ are diagonal and 
\[
\mathbf{F} = \left(
\begin{array}{cccc}
1 & 0 & 0 & 0
\\
1 & x_1 & 0 & 0
\\
1 & 0 & x_2 & 0
\\
1 & t & t & t^2
\end{array}
\right)
\] 
is a triangular matrix for $\xi \in \Xi_0$, the determinants of these matrices are the products of their entries on the diagonal.
Hence,
\[
\det(\mathbf{M}(\xi)) = w_0 w_1 w_2 w_3 x_1^2 \exp( - x_1) x_2^2 \exp( - x_2) t^4 \exp( - 2t - \rho t^2)
\]
and the weights as well as the single settings can be optimized separately.
As for all minimally supported designs the optimal weights are all equal to $1/p $ which is here $1/4$.
The contribution $x_j^2 \exp( - x_j)$ of the axial points is the same as in the corresponding marginal one-dimensional Poisson regression model with $\bdgr{\beta}_j=(0,-1)^\top$ and is optimized by $x_j = 2$, $j = 1, 2$. 
Finally, $t^4 \exp( - 2t - \rho t^2)$ is maximized by $t=(\sqrt{1+8\rho}-1)/(2\rho)$ for $\rho > 0$ and  $t=2$ for $\rho = 0$.
\end{proof}

\begin{proof}[Proof of Lemma~\ref{l:hyperbolicpath}]
The main idea behind this proof is to consider the deduced sensitivity function on contours of equal intensities.
For this we reparametrize the design region and use shifted and rescaled hyperbolic coordinates,
\[
x_1 = (v \exp(u) - 1)/\rho \qquad \textrm{ and } \qquad  x_2 = (v \exp(-u) - 1)/\rho ,
\]
where $v=\sqrt{(1+\rho x_1)(1+\rho x_2)}$ is the (shifted and scaled) hyperbolic distance and $u=\log(\sqrt{(1+\rho x_1)/(1+\rho x_2)})$ is the (shifted and scaled) hyperbolic angle in the case $\rho>0$.
The design region $\mathcal{X}=[0,\infty)^2$ is covered by $v \geq 1$ and $ |u| \leq \log(v) $.

With these coordinates, fixing  $v > 1$ returns a path parametrized in $u$ which intersects the diagonal at $u = 0$.
On each of these paths the intensity function $\lambda(\mathbf{x})$ is constant.
\begin{figure}[ht]
	\includegraphics[scale=.6]{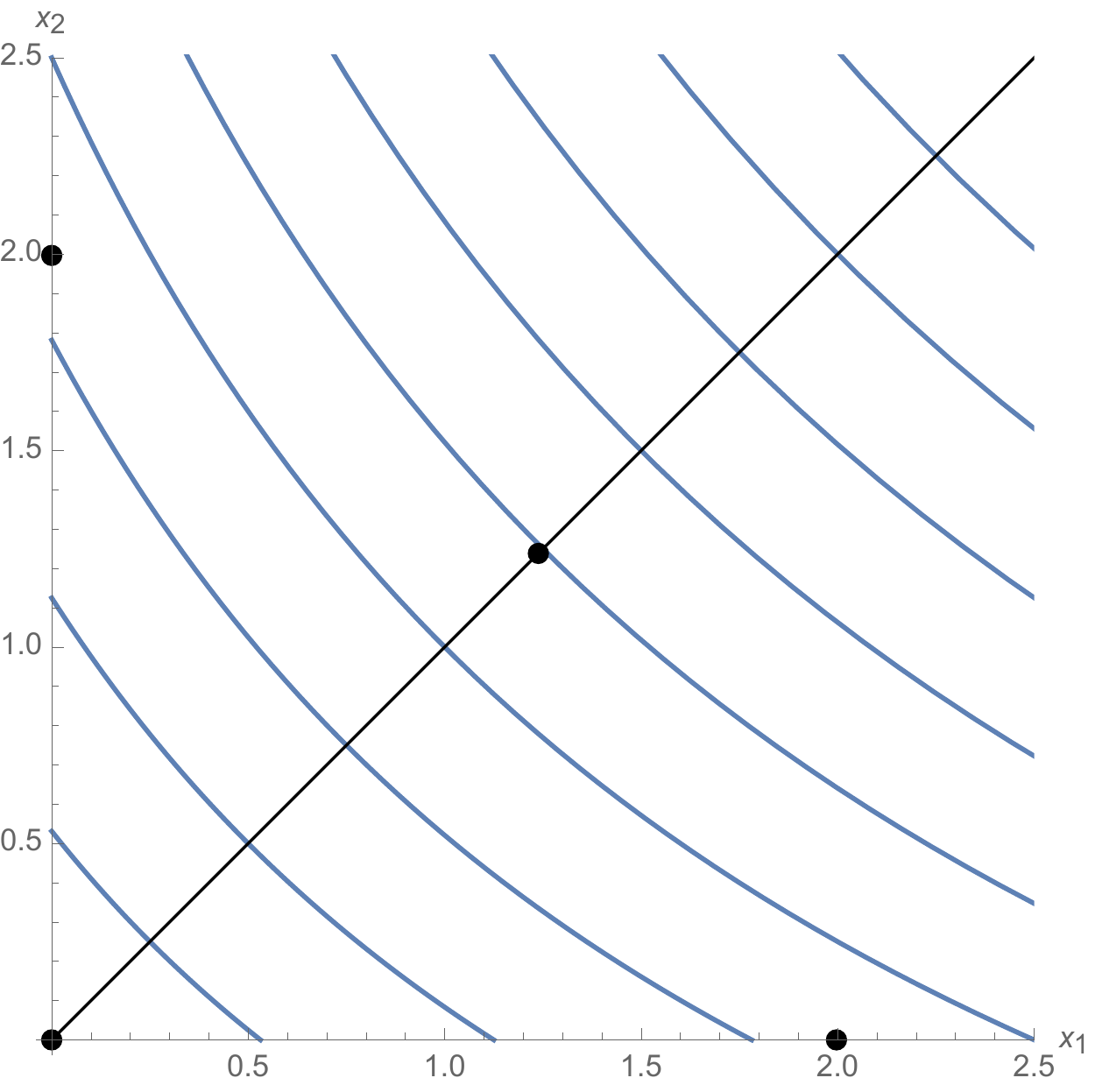}~	\includegraphics[scale=.6]{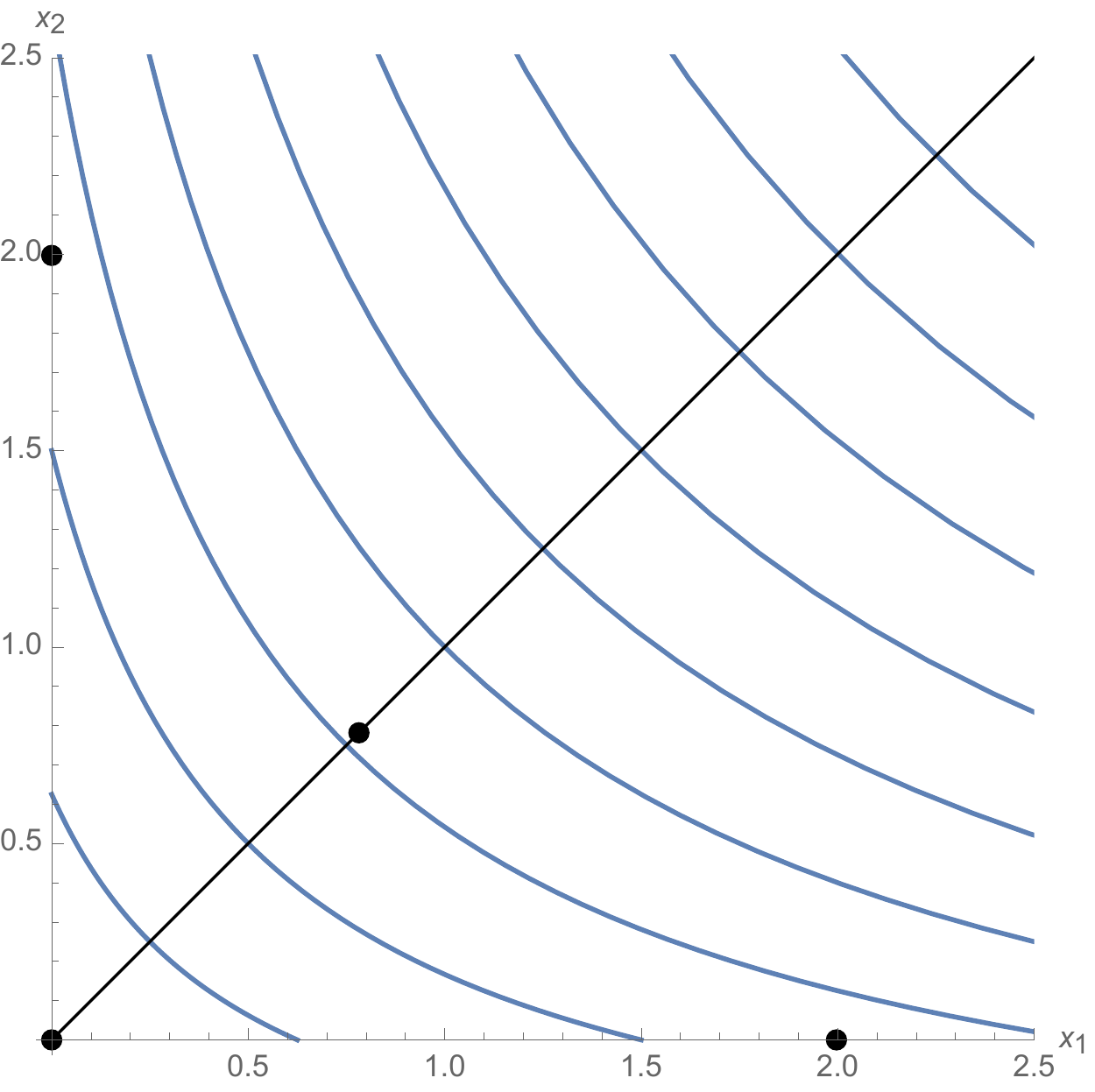}
	\caption{Lines of constant intensity for $ \rho=1/2 $ and $ \rho=2 $ with optimal design points}
\end{figure}

Because $\xi_t$ is invariant under permutation of $x_1$ and $x_2$, i.\,e.~sign change of $u$, the deduced sensitivity function $d(\mathbf{x};\xi_t)$ is symmetric in $u$, and we only have to consider the non-negative branch, $0 \leq u \leq \log(v)$. 
Using $\cosh(2u) = 2\cosh^2(u) - 1$, we observe that $d(\mathbf{x};\xi_t)$ is a quadratic polynomial in $\cosh(u) = (\exp(u) + \exp(-u))/2$ on each path.
Further, by the invariance of $\xi_t$, the information matrix and, hence, its inverse is invariant with respect to simultaneous exchange of the second and third columns and rows, respectively.
The leading coefficient of the quadratic polynomial can be written as $c(v) \mathbf{a}^{\top} \mathbf{M}(\xi)^{-1} \mathbf{a}$, where $\mathbf{a}=(0,-\rho,0,1)^\top$ and $c(v)$ is a positive constant depending on $v$. 
Since $\mathbf{M}(\xi)^{-1}$ is positive-definite, the leading coefficient is positive.
Now, any quadratic polynomial with positive leading coefficient attains its maximum over an interval on the boundary. 
This continues to hold if we compose the polynomial with a strictly monotonic function like $\cosh(u)$ on $[0,\log(v)]$.
Hence, on each path the maximum occurs at the diagonal ($u=0$, i.\,e.\ $x_1=x_2$) or on the boundary ($|u|=\log(v)$, i.\,e.\ $x_1=0$ or $x_2=0$). 
As the paths cover the whole design region, the statement of the Lemma follows for $\rho>0$.

In the case $\rho=0$ the contours of equal intensities degenerate to straight lines, where $x_1+x_2$ is constant.
Then the design region can be reparametrized by $x_1 = v+u$ and $x_2 = v - u$, where $v=(x_1 + x_2)/2 \geq 0$ is the (scaled directional $\ell_1$) distance from the origin and $u=(x_1 - x_2)/2$ is the (scaled $\ell_1$) distance from the diagonal, $|u| \leq v$.
Using similar arguments as for the case $\rho>0$ we can show that the sensitivity function restricted to each of these line segments for $v$ fixed is a symmetric polynomial in $u$ of degree $4$ with positive leading term.
Hence, also in the case $\rho=0$ the maximum of the sensitivity function can only be attained on the diagonal ($u=0$) or on the boundary ($|u|=v$) which completes the proof.
\end{proof}

\begin{proof}[Proof of Lemma~\ref{l:boundary}]
With the notation in the Proof of Lemma~\ref{l:design} the deduced sensitivity function can be written as
\begin{align}
\label{eq:sens_xt}
d(\mathbf{x};\xi_t) = \mathbf{f}(\mathbf{x})^{\top} \mathbf{F}^{-1} \mathbf{\Lambda}^{-1} (\mathbf{F}^{-1})^{\top} \mathbf{f}(\mathbf{x}) / p - 1 / \lambda(\mathbf{x}) ,
\end{align}
where
\[
\mathbf{F}^{-1} = \left(
\begin{array}{cccc}
1 & 0 & 0 & 0
\\
-1/2 & 1/2 & 0 & 0
\\
-1/2 & 0 & 1/2 & 0
\\
(t-1)/t^2 & -1/(2t) & -1/(2t) & 1/t^2
\end{array}
\right) ,
\] 
and similarly for the deduced sensitivity function $d_1(x;\xi_{-1}^*)$ of the locally $D$-optimal design $\xi_{-1}^*$ in the one-dimensional marginal model when $\bdgr{\beta}_1=(0,-1)^\top$.
For settings $\mathbf{x}=(x_1,0)$ we then obtain $d(\mathbf{x};\xi_t) = d_1(x_1;\xi_{-1}^*)$ by the relation between the quantities and matrices in both models and their special structure.
As $\xi_{-1}^*$ is $D$-optimal in the marginal model, its deduced sensitivity $d_1$ is bounded by zero by the equivalence theorem.
Hence, we obtain $d((x_1,0);\xi_t) \leq 0$ for all $x_1 \geq 0$.

For reasons of symmetry we also get $d((0,x_2);\xi_t) \leq 0$ for all $x_2 \geq 0$ which completes the proof.
\end{proof}

\begin{proof}[Proof of Lemma~\ref{l:diagonal}]
First note that the relation between $\rho$ and $t=(\sqrt{1+8\rho}-1)/(2\rho)$ is one-to-one such that conversely $\rho = (2-t)/t^2$.
Then, with the transformation $q=x/t$, the inequality to show in Lemma~\ref{l:diagonal} can be equivalently reformulated to
\begin{multline}
	\label{EQUI2}
	d(\mathbf{x};\xi_t) =\\
	 (q-1)^2 (q (t-1)-1)^2 + \frac{1}{2} \exp(2) t^2 (q-1)^2 q^2 + \exp(t+2) q^4 - \exp(2tq+(2-t)q^2) \leq 0 
\end{multline}
by using (\ref{eq:sens_xt}).
To prove the Lemma it is then sufficient to show that the inequality (\ref{EQUI2}) holds for all $0 \leq t \leq 2$ and all $q \geq 0$.

The idea behind the proof is to split the above function into a polynomial 
\[
		h_0(q,t)=\frac{1}{2} \exp(2) t^2 (q-1)^2 q^2+(q-1)^2 (q (t-1)-1)^2
\]
in $t$ and $q$ and a function 
\[
		 h_1(q,t) = \exp(2 q t + (2 - t) q^2) - \exp(t + 2) q^4
\]
involving the exponential terms such that  $d(\mathbf{x};\xi_t) = h_0(q,t) - h_1(q,t)$ and to find a suitable separating function $h_2(q,t)$ such that the inequalities $h_0(q,t) \leq h_2(q,t)$ and $h_2(q,t) \leq h_1(q,t)$ are easier to handle, where essentially methods for polynomials can be used for the former inequality while in the latter properties of exponential functions can be employed.

This function $h_2(q,t)$ will be defined piecewise in $q$ by
 \begin{equation*}
 h_2(q,t) = \left\{ 
\begin{array}{ll}
1 & \textrm{ for } q \leq q_0
\\
 \exp(t+2) (q-1)^2 q^2 & \textrm{ for } q > q_0
\end{array}  ,
 \right. 
 \end{equation*}
 where $q_0 = 3/5$, and the proof will be performed case-by-case. Figure \ref{fig:hoh1h2} visualizes this approach for selected values of $ t $.

\begin{figure}
	\includegraphics[scale=.6]{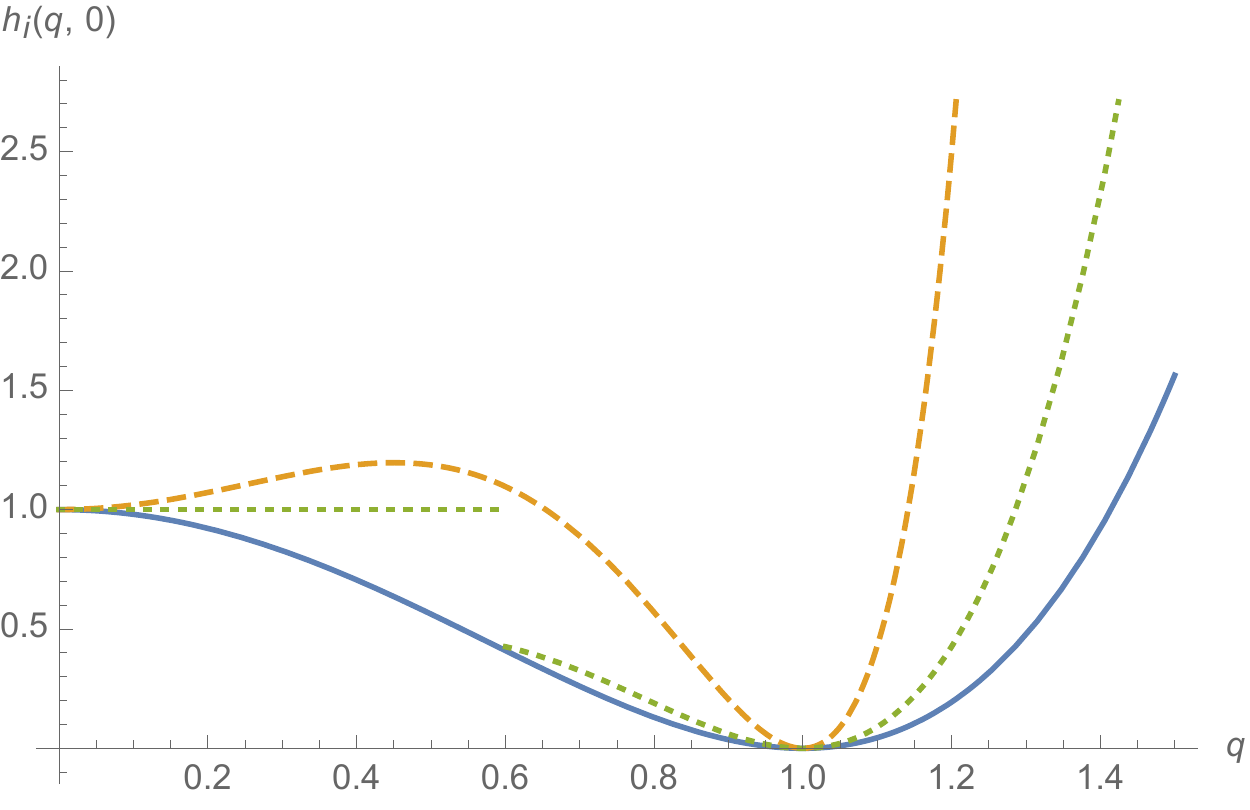}~\includegraphics[scale=.6]{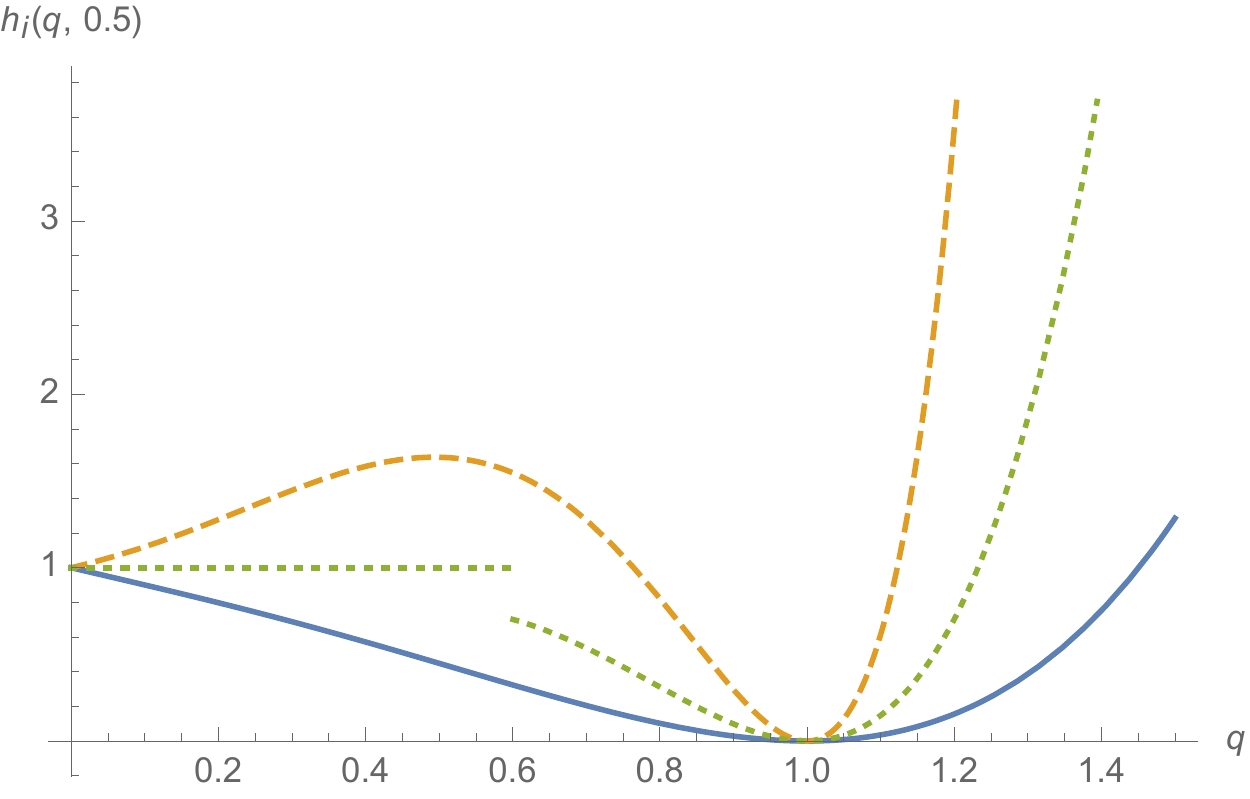}\\ 
	\includegraphics[scale=.6]{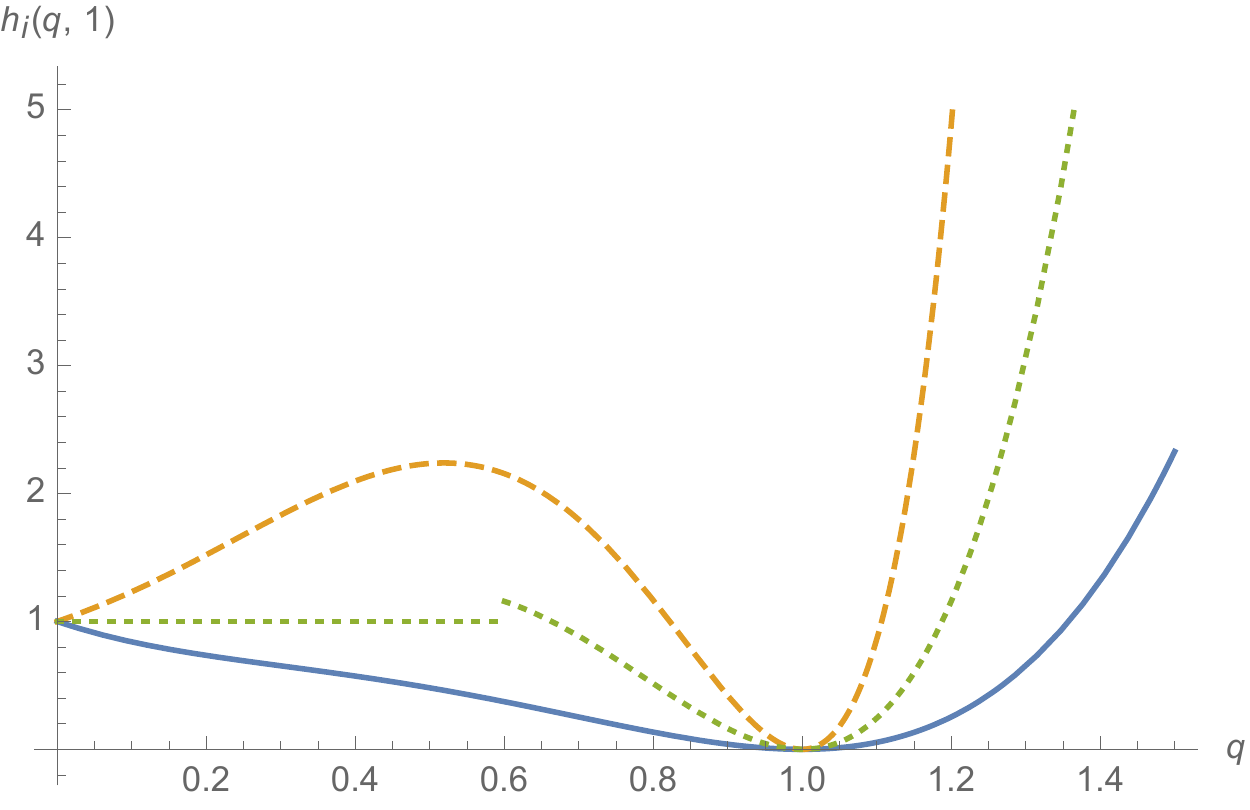}~\includegraphics[scale=.6]{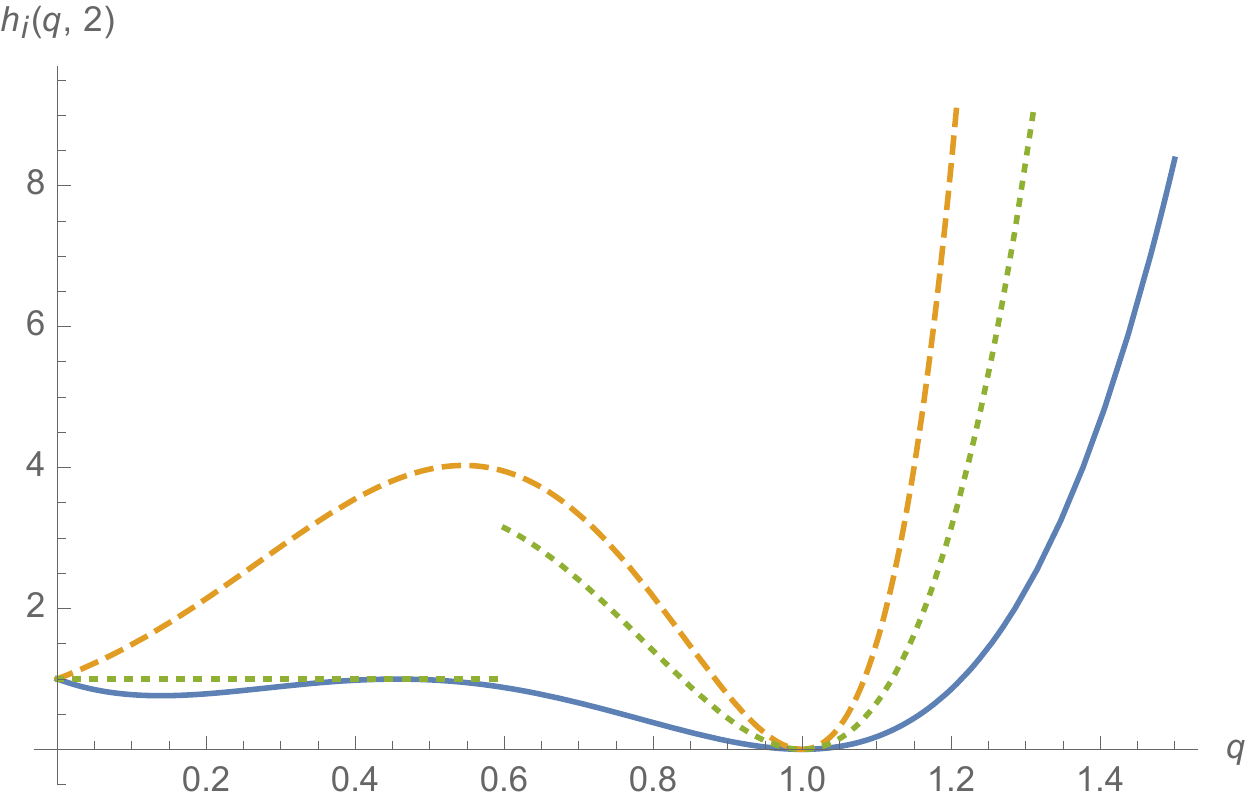}
	\caption{The functions $ h_0(q,t) $ (blue), $ h_1(q,t) $ (orange) and $ h_2(q,t) $ (green) for $ t= 0,1/2,1 $ and $ 2 $}\label{fig:hoh1h2}
\end{figure}

We start with the case $q \leq q_0$: 
The function $h_0(q,t)$  is a quadratic polynomial in $t$ with positive leading term.
Therefore its maximum over $0 \leq t \leq 2$ is attained at the end-points $t=0$ or $t=2$ of the interval.
Now, for $t=0$ we obtain
\[
	h_0(q,0) = (1 - q^2)^2 \leq 1
\]
for all $q \leq q_0$.

For $t=2$ 
\[
	h_0(q,2) =  (1 - q)^2 (2 \exp(2) q^2 + (1 - q)^2)
\]
 is a polynomial of degree $4$ in $q$ with positive leading term, $h_0(0,2) = 1$ and $h_0(1,2) = 0$.
 The polynomial has a local maximum 
 \begin{align*}
 	h_0(q_1,2) &= \frac{\exp(4) (3 \exp(1) - \sqrt{\exp(2) - 4})^2 (\exp(2) + 2 + \sqrt{\exp(4) - 4\exp(2)})}{8 (2 \exp(2) + 1)^3}\\
 	 &\approx 0.997
 \end{align*}
  at $q_1 = ( \exp(2) + 2 + \sqrt{\exp(4) - 4 \exp(2)})/(4 \exp(2) + 2) \approx 0.456.$
 This implies that $h_0(q,t) \leq 1$ for all $q \leq q_0$ and all $t \in [0,2]$.

Next we consider $h_1(q,t)$ as a function of $t$.
Its partial derivative with respect to $t$ is given by
\begin{equation}
\label{eq:partial_h1}  
\frac{\partial}{\partial t} h_1(q,t) = (2 - q) q \exp(q (2 t + (2 - t) q)) - q^4 \exp(t + 2) .
\end{equation}
If we compare the exponential terms, we see that
\begin{equation}
\label{eq:exponentlowbound}  
q (2 t +  (2 - t)q) - (t + 2) = - t (q - 1)^2 + 2 (q^2 - 1) \geq 4 (q - 1)
\end{equation}
 for all $ 0 \leq t \leq 2 $ uniformly in $q$.
Hence, the partial derivative (\ref{eq:partial_h1}) is non-negative if
\begin{align}
\label{eq:g1der2}
	q^{-3} (2 - q) \exp(4 (q - 1)) & \geq 1 .
\end{align}
To see this we notice
\[
 \frac{\partial}{\partial q} q^{-3} (2 - q) \exp(4 (q - 1)) = - 2 q^{-4}  (2 q^2 - 5 q + 3) \exp(4 (q - 1)) \leq 0
 \]
for $q \leq 1$ such that the expression on the left hand side of (\ref{eq:g1der2}) attains its minimum at $q =1$, where it is equal to $1$.
Combining the above results we obtain that $h_1(q,t)$ attains its minimum at $t = 0$ for all $q \leq 1$.  
It remains to show that  $h_1(q,0) = \exp(2 q^2) - \exp(2) q^4 \geq 1$ for all $q \leq q_0$. 
For this we check the derivative  
\[
\frac{\partial}{\partial q} h_1(q,0) = 4 q (\exp(2 q^2) - \exp(2) q^2) 
\]
with respect to $q$ which is positive for $0 < q < q_2 $ and negative for $ q_2 < q \leq q_0 $. where $q_2 \approx 0.451 $.
Hence, evaluating $h_1(q,0)$ a the end-points of the relevant interval, $h_1(0,0) = 1$ and $h_1(q_0,0) \approx 1.097 $, we get $h_1(q,0) \geq 1$ which finally implies $h_0(q,t) \leq 1 \leq h_1(q,t)$ for all $q \leq q_0$ and all $0 \leq t \leq 2$.

For the case $q > q_0$ the condition $h_0(q,t) \leq h_2(q,t)$  is equivalent to
\begin{equation}
\label{eq:taylor1}
(q (t - 1) - 1)^2 \leq  q^2 \exp(2)  (\exp(t) - t^2/2).
\end{equation}
By the exponential series expansion, $ \exp(t) \geq 1 + t + t^2/2$ for $ t\ge 0 $, the right hand side is bounded from below by $(t + 1)  q^2 \exp(2)$, and for (\ref{eq:taylor1}) to hold it is sufficient to show
\begin{align}
\label{eq:taylor2}
(\exp(2) (t + 1) - (t - 1)^2) q^2  + 2 (t - 1) q - 1 & \geq 0 .
\end{align}
The derivative of this expression with respect to $ q $ equals
\[
2 (\exp(2) (t + 1) - (t - 1)^2) q + 2 (t - 1) \geq  \exp(2) (t + 1) -t^2+4t-3 \geq 0
\]
for $q \geq 1/2$ and all $0 \leq t \leq 2$.
Hence, the expression in (\ref{eq:taylor2}) itself is bounded from below by its value at $q_0 = 3/5$, which is approximately 0.1001.

This establishes $h_0(q,t) \leq h_2(q,t)$ for all $q > q_0$ and all $0 \leq t \leq 2$.

Finally, the condition $h_2(q,t) \leq h_1(q,t)$  is equivalent to
\[
(1 - q)^2 q^2 + q^4 \leq \exp(q (2 t + (2 - t) q) - (t + 2)) .
\]
Again, by $q (2 t+ (2-t)q)-(t+2) \geq 4(q-1)$ for all $ 0 < t < 2 $, see (\ref{eq:exponentlowbound}), it is sufficient to show
\begin{align}
\label{eq:minorante}
((1 - q)^2 q^2 + q^4) \exp(4(1 - q)) \leq 1  
\end{align}
for all $ q \geq 0 $.
The derivative of this expression equals
\[
2 (1 - 2 q)^2 (1 - q) q \exp(4 (1 - q)) .
\]
Hence, for $q \geq 0$ the expression in (\ref{eq:minorante}) attains its maximum at $ q=1 $, where it is equal to $1$. 
This implies $h_2(q,t) \leq h_1(q,t)$ for all $q > q_0$ and all $0 \leq t \leq 2$ which completes the proof.
\end{proof}

\begin{proof}[Proof of Theorem~\ref{thm:k-dim_1st-order}]
Here we only give a sketch of the proof.
As in the Proof of Lemma~\ref{l:hyperbolicpath} we see that the paths of equal intensity constitute hyper-planes intersecting the design region at equilateral simplices.
On each straight line within these simplices the sensitivity function is a polynomial of degree four with positive leading term.
Hence, following the idea of the proofs in Farrell et al.\ \cite{FKW1967} we can conclude by symmetry considerations with respect to permutation of the entries in $\mathbf{x}$ we can conclude that the sensitivity function may attain a maximum in the interior of the design region only at the diagonal, where all entries in $\mathbf{x}$ are equal ($x_1=x_2=...=x_k=x$) and in the relative interior of each $j$-dimensional face of the design region on the respective diagonal, where all the $j$ non-zero entries of $\mathbf{x}$ are equal to some $x$, $2 \leq j \leq k$.

Similar to the Proof of Lemma~\ref{l:boundary} on each face the deduced sensitivity function is equal to its counterpart for the $D$-optimal design in the two-dimensional marginal model on that face and is, thus, bounded by $0$.

Finally, to derive the deduced sensitivity function on the diagonals we specify the essential design matrix $\mathbf{F}$ and its inverse
\[
\mathbf{F} = \left(
\begin{array}{ccc}
1 & \mathbf{0} & \mathbf{0} 
\\
\mathbf{1}_k & \mathbf{I}_k & \mathbf{0} 
\\
 \mathbf{1}_{C(k,2)} & \mathbf{S}_2 & \mathbf{I}_{C(k,2)}
\end{array}
\right) \mathbf{A}
\qquad \textrm{ and } \qquad
\mathbf{F}^{-1} = \mathbf{A}^{-1} \left(
\begin{array}{ccc}
1 & \mathbf{0} & \mathbf{0} 
\\
-\mathbf{1}_k & \mathbf{I}_k & \mathbf{0} 
\\
 \mathbf{1}_{C(k,2)} & -\mathbf{S}_2 & \mathbf{I}_{C(k,2)}
\end{array}
\right) ,
\] 
where $\mathbf{A}=\mathrm{diag}(1, 2\,\mathbf{1}_k^\top, 4\, \mathbf{1}_{C(k,2)}^\top)$ is a diagonal matrix related to the product of the non-zero coordinates of the design points, $\mathbf{1}_m$ is a $m$-dimensional vector with all entries equal to $1$,  $\mathbf{I}_m$ is the $m \times m$ identity matrix, $C(m,n)$ denotes binomial coefficient $\binom{m}{n}$, and $\mathbf{S}_2$ is the incidence matrix of a balanced incomplete block design (BIBD) for $k$ varieties and all $C(k,2)$ blocks of size $2$.
Then by \eqref{eq:sens_xt} the deduced sensitivity function equals
\[
(C(j,2) q^2 - j q + 1)^2 + j \exp(2) ((j-1) q^2 - q)^2 + C(j,2) \exp(4) q^4 - \exp(2 j q) 
\]
on the diagonals of all $j$-dimensional faces, $j<k$, and the interior diagonal for $j=k$, where $q=x/2$ as in the Proof of Lemma~\ref{l:diagonal}.
By using \Mathematica and a power series expansion of order $5$ for the term $\exp(2 k q)$ the above expression can be seen not to exceed $0$ for all $q \geq 0$ which establishes the local $D$-optimality in view of the equivalence theorem.
\end{proof}

\begin{proof}[Proof of Theorem~\ref{thm:k-dim_2nd-order}]
The proof goes along the lines of the Proof of Theorem~\ref{thm:k-dim_1st-order}.
The essential design matrix $\mathbf{F}$ and its inverse are specified as
\begin{align*}
\mathbf{F} &= \left(
\begin{array}{cccc}
1 & \mathbf{0} & \mathbf{0}  & \mathbf{0} 
\\
\mathbf{1}_k & \mathbf{I}_k & \mathbf{0}  & \mathbf{0} 
\\
 \mathbf{1}_{C(k,2)}&\mathbf{S}_2 & \mathbf{I}_{C(k,2)}  & \mathbf{0} 
\\
 \mathbf{1}_{C(k,3)} & \mathbf{S}_3 & \mathbf{S}_{23} & \mathbf{I}_{C(k,3)} 
\end{array}
\right) \mathbf{A},\\
\mathbf{F}^{-1} &= \mathbf{A}^{-1} \left(
\begin{array}{cccc}
1 & \mathbf{0} & \mathbf{0} & \mathbf{0} 
\\
-\mathbf{1}_k & \mathbf{I}_k & \mathbf{0} & \mathbf{0} 
\\
 \mathbf{1}_{C(k,2)} & -\mathbf{S}_2 & \mathbf{I}_{C(k,2)} & \mathbf{0} 
\\
-\mathbf{1}_{C(k,3)} & \mathbf{S}_3 & -\mathbf{S}_{23} & \mathbf{I}_{C(k,3)} 
\end{array}
\right) ,
\end{align*}
where now $\mathbf{A}=\mathrm{diag}(1, 2\, \mathbf{1}_k^\top, 4\, \mathbf{1}_{C(k,2)}^\top, 8\, \mathbf{1}_{C(k,3)}^\top)$, $\mathbf{S}_3$ is the incidence matrix of a BIBD for $k$ varieties and all $C(k,3)$ blocks of size $3$, and $\mathbf{S}_{23}$ is the (generalized) $C(k,3) \times C(k,2)$ incidence matrix which relates all blocks of size $2$ to those blocks of size $3$ ln which their components are included.
Then the deduced sensitivity function equals
\begin{multline*}
	(C(j,3) q^3 - C(j,2) q^2 + j q - 1)^2+j \exp(2) ((C(j,2) - j + 1) q^3 - (j - 1) q^2 + q)^2\\ + C(j,2) \exp(4) ((j - 2) q^3 - q^2)^2+C(j,3) \exp(6) q^6  - \exp(2 j q)
\end{multline*}
on the diagonals, where $q=x/2$.
By using \Mathematica and a power series expansion of order $9$ for the term $\exp(2 k q)$ the above expression can be seen not to exceed $0$ for all $q \geq 0$ which establishes the local $D$-optimality.
\end{proof}

\bibliographystyle{plain}
\bibliography{bibliography}

\end{document}